\documentclass[12pt,reqno]{amsart}
\calclayout
\usepackage{amssymb}
\usepackage{mathtools}
\usepackage[hidelinks]{hyperref}
\newtheorem{Thm}{Theorem}[section]

\newtheorem{Lem}[Thm]{Lemma}
\newtheorem{Prop}[Thm]{Proposition}
\newtheorem{Cor}[Thm]{Corollary}

\newtheorem{Conj}[Thm]{Conjecture}

\newcommand{\CC}{\mathbb{C}}
\newcommand{\bk}{\boldsymbol{k}}
\newcommand{\bl}{\boldsymbol{l}}
\newcommand{\QQ}{\mathbb{Q}}

\newcommand{\ZZ}{\mathbb{Z}}
\newcommand{\lf}{\leftarrow}
\newcommand{\up}{\uparrow}
\newcommand{\down}{\downarrow}
\newcommand{\wt}{\mathrm{wt}}
\newcommand{\ol}{\overline}

\begin{document}
\title[Hirose's relation for finite multiple harmonic $q$-series]{An analogue of Hirose's relation for finite multiple harmonic $q$-series at roots of unity}
\author{Takumi Maesaka}
\email{nozaki.takumi.912@s.kyushu-u.ac.jp}
\subjclass{11M32, 11R18, 05A30}
\keywords{finite multiple harmonic $q$-series, roots of unity, cyclic sum formula, connected sums}
\begin{abstract}
Recently, Hirose proved an analogue of the linear part of Kawashima's relations for refined symmetric multiple zeta values. In this paper, we establish an analogue of Hirose's relation for finite multiple harmonic $q$-series at roots of unity. As an application, we  prove the cyclic sum conjecture proposed by Kh. Hessami Pilehrood, T. Hessami Pilehrood, and R. Tauraso.
\end{abstract}
\maketitle

\section{Introduction}

For an index $\bk\coloneqq(k_1,\dots,k_r)\in (\ZZ_{>0})^r$, a positive integer $N$, and a complex number $q$ with $q^n\neq 1$ for $0<n<N$, we define
\[
z_N(\bk;q)\coloneqq\sum_{0<n_1<\cdots<n_r< N}\prod_{i=1}^r\frac{q^{(k_i-1)n_i}}{[n_i]_q^{k_i}},
\]
where
\[
[n]_q\coloneqq\frac{1-q^n}{1-q}
\]
is the $q$-integer. For the primitive $N$-th root of unity $\zeta_N\coloneqq e^{\frac{2\pi i}{N}}$, we use the notation
\[
z_N(\bk)\coloneqq z_N(\bk;\zeta_N).
\]

Hirose \cite{h1} introduced refined symmetric multiple zeta values (RSMZVs) and proved that they coincide with $\xi$-values of Bachmann, Takeyama, and Tasaka \cite{btt}. In other words,
\[
\lim_{N\to\infty} z_N(\bk)=\zeta_{\mathcal{RS}}(\bk).
\]
Moreover, Bachmann, Takeyama, and Tasaka \cite{btt} showed that
\[
(z_p(\bk)\,\mathrm{mod}\,\mathfrak{p})_p=\zeta_{\mathcal{A}}(\bk),
\]
where $\mathfrak{p}\coloneqq (1-\zeta_p)$ is the prime ideal of $\ZZ[\zeta_p]$ generated by $1-\zeta_p$. Kaneko and Zagier \cite{kz} introduced the finite multiple zeta values
(FMZVs), which are defined as follows:
\[
\zeta_{\mathcal{A}}(k_1,\dots,k_r)\coloneqq\left(\sum_{0<n_1<\cdots<n_r<p}\frac 1{n_1^{k_1}\cdots n_r^{k_r}}\,\mathrm{mod}\,p\right)_p\in \left(\prod_{p:\mathrm{prime}}\ZZ/p\ZZ\right)\bigg/\left(\bigoplus_{p:\mathrm{prime}}\ZZ/p\ZZ\right).
\]
Thus, the sums $z_N$ introduced by Bachmann, Takeyama,
and Tasaka simultaneously generalize RSMZVs and FMZVs.

For an index $\bk\coloneqq(k_1,\dots,k_r)\in (\ZZ_{>0})^r$ with $k_r\geq 2$, multiple zeta values (MZVs) are defined by
\[
\zeta(\bk)\coloneqq \sum_{0<n_1<\cdots<n_r}\frac 1{n_1^{k_1}\cdots n_r^{k_r}}.
\]

Kawashima \cite{k} established a family of algebraic relations among multiple zeta values, now known as Kawashima's relations. Their linear part is known to imply several important relations among multiple zeta values, including duality, the derivation relations, Ohno's relations, and the cyclic sum formula.

Recently, Hirose \cite{h2} used the multitangent functions introduced by Bouillot \cite{b} to establish an analogue of the linear part of Kawashima's relations for refined symmetric multiple zeta values.

\begin{Thm}[Hirose {\cite[Theorem 9]{h2}}]
For non-empty indices $\bk=(k_1,\dots,k_r)$ and $\bl=(l_1,\dots,l_s)$ with $k_1,k_r, l_1,l_s\geq 2$, the identity
\[
\zeta_{\mathcal{RS}}(({}_{\down}(\bk*\bl)_{\down})^{\vee})=0
\]
holds.
\end{Thm}
Here, $*$ denotes the harmonic product, and $\bk^\vee$ is the Hoffman dual index of $\bk$ (for the definitions of these symbols, see Section 2). We generalize this relation to finite multiple harmonic $q$-series at roots of unity. The following is our main result.

\begin{Thm}\label{thm:main}
For non-empty indices $\bk=(k_1,\dots,k_r)$ and $\bl=(l_1,\dots,l_s)$ with $k_1,k_r, l_1,l_s\geq 2$, the identity
\[
z_N(({}_{\down}(\bk*_{\zeta_N}\bl)_{\down})^{\vee})=0
\]
holds.
\end{Thm}

Here, \(\ast_q\) denotes a \(q\)-analogue of the harmonic product. Letting $N\to\infty$ in Theorem \ref{thm:main} yields Hirose's result and hence provides another proof of it. Similarly, setting $N=p$ and reducing modulo $\mathfrak{p}$ yields the following corollary.

\begin{Cor}
For non-empty indices $\bk=(k_1,\dots,k_r)$ and $\bl=(l_1,\dots,l_s)$ with $k_1,k_r, l_1,l_s\geq 2$, the identity
\[
\zeta_{\mathcal{A}}(({}_{\down}(\bk*\bl)_{\down})^{\vee})=0
\]
holds.
\end{Cor}

In \cite{hphpt}, Hessami Pilehrood, Hessami Pilehrood, and Tauraso conjectured the following.
\begin{Conj}[Hessami Pilehrood--Hessami Pilehrood--Tauraso, Conjecture 1.1 (i)]\label{conj:main}
Let $d_0,\dots,d_t$ be non-negative integers and $N$ be an integer greater than $r\coloneqq \sum_{j=0}^td_j+2t$. Then
\begin{align*}
\sum_{j=0}^tz_N\left(\{1\}^{d_j},2,\{1\}^{d_{j+1}},2,\dots,2,\{1\}^{d_{j+t}}\right)&=\frac{(-1)^t}{N}\binom{N+t}{r+1}(1-\zeta_N)^r.
\end{align*}
In the sum above, it is understood that $d_j=d_k$ if $j\equiv k$ modulo $t+1$. Also, $\{1\}^n$ denotes $n$ repetitions of $1$.
\end{Conj}

In the case of multiple zeta values, Tanaka and Wakabayashi \cite{tw} showed that the cyclic sum formula can be derived algebraically from Kawashima's relations. By slightly modifying their result and combining it with Theorem \ref{thm:main}, we can prove Conjecture \ref{conj:main}.

\begin{Thm}\label{thm:prove_conj}
Conjecture \ref{conj:main} is true.
\end{Thm}
As a consequence, we obtain both the symmetric and finite versions.

\begin{Cor}
Let $d_0,\dots,d_t$ be non-negative integers, and set $r\coloneqq \sum_{j=0}^td_j+2t$. Then
\begin{align*}
\sum_{j=0}^t\zeta_{\mathcal{RS}}\left(\{1\}^{d_j},2,\{1\}^{d_{j+1}},2,\dots,2,\{1\}^{d_{j+t}}\right)&=\frac{(-1)^t(-2\pi i)^r}{(r+1)!},\\
\sum_{j=0}^t\zeta_{\mathcal{A}}\left(\{1\}^{d_j},2,\{1\}^{d_{j+1}},2,\dots,2,\{1\}^{d_{j+t}}\right)&=0.
\end{align*}
In the sums above, it is understood that $d_j=d_k$ if $j\equiv k$ modulo $t+1$.
\end{Cor}

The remainder of this paper is organized as follows. In Section 2, we prove Theorem \ref{thm:main} using the connected sum method introduced by Seki and Yamamoto \cite{sy}. In Section 3, we prove Theorem \ref{thm:prove_conj} by combining Theorem \ref{thm:main} with an algebraic identity in the quasi-shuffle algebra.
\section{Proof of Theorem \ref{thm:main}}

\subsection{Connected sums}
We use the arrow notation introduced by Seki \cite{s} and the connected sum method introduced by Seki and Yamamoto \cite{sy}. For an index $\bk\coloneqq (k_1,\dots,k_r)$, we define
\begin{align*}
\bk_{\up}&\coloneqq(k_1,\dots,k_{r-1},k_r+1),\quad \bk_{\down}\coloneqq(k_1,\dots,k_{r-1},k_r-1),\\
{}_{\up}\bk&\coloneqq(k_1+1,k_2,\dots,k_r),\quad {}_{\down}\bk\coloneqq(k_1-1,k_2,\dots,k_{r-1},k_r),\\
\bk_{\to}&\coloneqq(k_1,\dots,k_{r-1},k_r,1),\quad {}_{\lf}\bk\coloneqq(1,k_1,k_2,\dots,k_r).
\end{align*}
We now define the connector and the connected sums.
Throughout this subsection, let \(q,s,t\in\CC\) satisfy
\[
|q|<1,\qquad |s|<1,\qquad |t|<1.
\]
For non-negative integers \(n,m\), let
\begin{align*}
(a;q)_n&\coloneqq\prod_{k=0}^{n-1}(1-aq^k),\qquad(a;q)_{\infty}\coloneqq\lim_{n\to\infty}(a;q)_n,\\
C(n,m)&\coloneqq\frac{(sq;q)_n(tq;q)_m}{(stq;q)_{n+m}}t^ns^mq^{nm}.
\end{align*}
We call an index $\bk=(k_1,\dots,k_a)$ admissible if $a=0$ or $k_a\geq 2$. For indices $\bk\coloneqq (k_1,\dots,k_a), \bl\coloneqq (l_1,\dots,l_b)$, we define
\begin{align*}
\zeta_q(\bk;s,t)&\coloneqq\sum_{0<n_1<\cdots<n_a}\left(\prod_{i=1}^a\frac{(sq^{n_i})^{k_i-1}}{(1-sq^{n_i})^{k_i}}\right)\frac{(sq;q)_{n_a}}{(stq;q)_{n_a}}t^{n_a},\\
Z_{q,s,t}(\bk;\bl)&\coloneqq\sum_{\substack{0=n_0<n_1<\cdots<n_a\\0=m_0< m_1<\cdots<m_b}}\left(\prod_{i=1}^a\frac{(sq^{n_i})^{k_i-1}}{(1-sq^{n_i})^{k_i}}\right)C(n_a,m_b)\left(\prod_{i=1}^b\frac{(tq^{m_i})^{l_i-1}}{(1-tq^{m_i})^{l_i}}\right).
\end{align*}

From the definition, the symmetry

\[
Z_{q,s,t}(\bk;\bl)=Z_{q,t,s}(\bl;\bk)
\]
holds, and in particular,
\begin{align*}
Z_{q,s,t}(\bk;\varnothing)&=\zeta_q(\bk;s,t),\\
Z_{q,s,t}(\varnothing;\bk)&=\zeta_q(\bk;t,s).
\end{align*}

The following relations are called transport relations in the context of the connected sum method.
\begin{Lem}[Transport relations]
For indices $\bk$ and $\bl$, the identities
\begin{align*}
Z_{q,s,t}(\bk_{\to};\bl)&=Z_{q,s,t}(\bk;\bl_{\up}),\\
Z_{q,s,t}(\bk_{\up};\bl)&=Z_{q,s,t}(\bk;\bl_{\to})
\end{align*}
hold.
\end{Lem}

\begin{proof}
By symmetry, it suffices to prove the following identity
\[
\sum_{n<k}\frac{1}{1-sq^k}C(k,m)=C(n,m)\frac{tq^m}{1-tq^m}.
\]
This follows from the telescoping identity
\begin{align*}
\sum_{n<k}\frac 1{1-sq^k}C(k,m)&=\sum_{n<k}(C(k-1,m)-C(k,m))\frac{tq^m}{1-tq^m}\\
&=C(n,m)\frac{tq^m}{1-tq^m}.
\end{align*}
The final equality uses the fact that $\lim_{n\to \infty}C(n,m)=0$.
\end{proof}
Every admissible index $\bk$ can be written uniquely in the form
\[
\bk=(\{1\}^{a_1-1},b_1+1,\dots,\{1\}^{a_s-1},b_s+1)
\]
where $a_1,\dots,a_s,b_1,\dots,b_s$ are positive integers. The dual index of $\bk$ is defined by
\[
\bk^{\dagger}\coloneqq(\{1\}^{b_s-1},a_s+1,\dots,\{1\}^{b_1-1},a_1+1).
\]

For an admissible index $\bk\coloneqq (k_1,\dots,k_a)$, applying the transport relations $k_1+\cdots+k_a$ times yields the following proposition.

\begin{Prop}[Duality formula]\label{prop:duality}
For an admissible index $\bk$, the identity
\[
\zeta_q(\bk;s,t)=\zeta_q(\bk^{\dagger};t,s)
\]
holds.
\end{Prop}

\subsection{Algebraic setup}

We define the harmonic product \(\ast\) recursively as a bilinear
product on $\oplus_{\bk:\mathrm{index}}\QQ\bk$ by
\begin{align*}
\bk*\varnothing&=\varnothing*\bk=\bk,\\
\bk*\bl&=(k_1,{}_{-}\bk*\bl)+(l_1,\bk*{}_{-}\bl)+(k_1+l_1,{}_{-}\bk*{}_{-}\bl),\quad r\geq 1,s\geq 1
\end{align*}
where ${}_{-}\bk\coloneqq(k_2,\dots,k_{r})$. Similarly, the $q$-harmonic product is defined by
\begin{align*}
\bk*_q\varnothing&=\varnothing*_q\bk=\bk,\\
\bk*_q\bl&=(k_1,{}_{-}\bk*_q\bl)+(l_1,\bk*_q{}_{-}\bl)\\
&\quad +(k_1+l_1,{}_{-}\bk*_q{}_{-}\bl)+(1-q)(k_1+l_1-1,{}_{-}\bk*_q{}_{-}\bl),\quad r\geq 1,s\geq 1.
\end{align*}
We put $*_+\coloneqq *_0$, and define
\[
\ol{z}_N(\bk;q)\coloneqq(1-q)^{-\wt(\bk)}z_N(\bk;q), \ol{z}_N(\bk)\coloneqq(1-\zeta_N)^{-\wt(\bk)}z_N(\bk).
\]
For any index $\bk\coloneqq(k_1,\dots,k_r)$, we write
\[
\overleftarrow{\bk}\coloneqq(k_r,k_{r-1},\dots,k_1)
\]
and define the Hoffman dual by 
\[
\bk^{\vee}\coloneqq\overleftarrow{((\bk_{\up})^{\dagger})_{\down}}.
\]
For an admissible index \(\bk\), the defining series
\(\zeta_q(\bk;s,t)\) converges absolutely also at the
boundary specializations \(s=1\) or \(t=1\). Therefore,
Proposition 2.2 extends to these specializations by continuity. For an admissible index $\bk\coloneqq (k_1,\dots,k_a)$, we define
\begin{align*}
\zeta_q(\bk;t)&\coloneqq\zeta_q(\bk;t,1)=\sum_{0<n_1<\cdots<n_a}\prod_{i=1}^a\frac{(tq^{n_i})^{k_i-1}}{(1-tq^{n_i})^{k_i}}\\
\widetilde{\zeta}_q(\bk;t)&\coloneqq\zeta_q(\bk;1,t)=\sum_{0<n_1<\cdots<n_a}\left(\prod_{i=1}^a\frac{q^{n_i(k_i-1)}}{(1-q^{n_i})^{k_i}}\right)\frac{(q;q)_{n_a}}{(tq;q)_{n_a}}t^{n_a}.
\end{align*}

Then $\zeta_q(\bk;t)$ satisfies an analogue of the harmonic product:
\[
\zeta_q(\bk;t)\zeta_q(\bl;t)=\zeta_q(\bk*_+\bl;t),
\]
because
\[
\frac{(tq^n)^{k-1}}{(1-tq^n)^k}\cdot\frac{(tq^n)^{l-1}}{(1-tq^n)^l}=\frac{(tq^n)^{k+l-1}}{(1-tq^n)^{k+l}}+\frac{(tq^n)^{k+l-2}}{(1-tq^n)^{k+l-1}}.
\]
By Proposition 2.2 and the preceding continuity argument, we have
\[
\widetilde{\zeta}_q(\bk;t)
 =\zeta_q(\bk^{\dagger};t).
\]
These identities yield the following corollary.
\begin{Cor}\label{cor:product}
For admissible indices $\bk,\bl$,
\[
\widetilde{\zeta}_q(\bk;t)\widetilde{\zeta}_q(\bl;t)=\widetilde{\zeta}_q((\bk^{\dagger}*_+\bl^{\dagger})^{\dagger};t)
\]
holds.
\end{Cor}

\begin{Lem}\label{lem:product}
Let \(n,m\) be positive integers, and assume that
\(|q|<1\) and \(|t|<1\). Then
\[
\sum_{n+m\leq r}\frac{(q;q)_{r-n-1}(q;q)_{r-m-1}}{(q;q)_r(q;q)_{r-n-m}}q^r\frac{(q;q)_r}{(tq;q)_r}t^r=\frac{q^{n+m}}{(1-q^n)(1-q^m)}\frac{(q;q)_n(q;q)_m}{(tq;q)_n(tq;q)_m}t^{n+m}
\]
holds.
\end{Lem}

\begin{proof}
A direct calculation gives
\begin{align*}
&\sum_{n+m\leq r}\frac{(q;q)_{r-n-1}(q;q)_{r-m-1}}{(q;q)_r(q;q)_{r-n-m}}q^r\frac{(q;q)_r}{(tq;q)_r}t^r\\
&=\sum_{0\leq r}\frac{(q;q)_{r+m-1}(q;q)_{r+n-1}}{(tq;q)_{r+n+m}(q;q)_{r}}(tq)^{r+n+m}\\
&=(tq)^{n+m}\frac{(q;q)_{m-1}(q;q)_{n-1}}{(tq;q)_{n+m}}\sum_{0\leq r}\frac{(q^m;q)_r(q^n;q)_r}{(tq^{n+m+1};q)_r(q;q)_{r}}(tq)^r.
\end{align*}
By Heine's summation formula (see Gasper--Rahman \cite[\S 1.5]{gr})
\begin{align*}
\sum_{0\leq r}\frac{(q^m;q)_r(q^n;q)_r}{(tq^{n+m+1};q)_r(q;q)_{r}}(tq)^r&=\frac{(tq^{n+1};q)_{\infty}(tq^{m+1};q)_{\infty}}{(tq^{n+m+1};q)_{\infty}(tq;q)_{\infty}}\\
&=\frac{(tq;q)_{n+m}}{(tq;q)_n(tq;q)_m}.
\end{align*}
Substituting this identity proves the lemma.
\end{proof}

Theorem \ref{thm:main} can be rewritten as follows.
\begin{Thm}\label{thm:main0}
For non-empty indices $\bk=(k_1,\dots,k_r)$ and $\bl=(l_1,\dots,l_s)$ with $k_1,k_r, l_1,l_s\geq 2$, the identity
\[
\ol{z}_N(({}_{\down}(\bk*_{+}\bl)_{\down})^{\vee})=0
\]
holds.
\end{Thm}
\begin{proof}
Let $(k_1',\dots,k_{a'}')=\bk'\coloneqq(\overleftarrow{\bk})^{\dagger}$ and $(l_1',\dots,l_{b'}')=\bl'\coloneqq(\overleftarrow{\bl})^{\dagger}$. By Corollary \ref{cor:product} and Lemma \ref{lem:product}, we obtain
\begin{align*}
&\widetilde{\zeta}_q((\overleftarrow{\bk}*_+\overleftarrow{\bl})^{\dagger};t)\\
&=\widetilde{\zeta}_q(\bk';t)\widetilde{\zeta}_q(\bl';t)\\
&=\sum_{\substack{0<n_1<\cdots<n_{a'}\\0<m_1<\cdots<m_{b'}\\n_{a'}+m_{b'}\leq r}}\left(\prod_{i=1}^{a'}\frac{q^{n_i(k_i'-1)}}{(1-q^{n_i})^{k_i'}}\right)\left(\prod_{i=1}^{b'}\frac{q^{m_i(l_i'-1)}}{(1-q^{m_i})^{l_i'}}\right)\\
&\qquad\cdot\frac{1-q^{n_{a'}}}{q^{n_{a'}}}\frac{1-q^{m_{b'}}}{q^{m_{b'}}}\frac{(q;q)_{r-n_{a'}-1}(q;q)_{r-m_{b'}-1}}{(q;q)_r(q;q)_{r-n_{a'}-m_{b'}}}q^r\frac{(q;q)_r}{(tq;q)_r}t^r.
\end{align*}
The identity
\[
\frac{(q;q)_r}{(tq;q)_r}t^r=(q;q)_r t^r+O(t^{r+1})
\]
together with the fact that \((q;q)_r\) is a unit in
\(\CC[[q]]\) shows that the family
\[
\left\{\frac{(q;q)_r}{(tq;q)_r}t^r\right\}_{r\geq0}
\]
forms a topological \(\CC[[q]]\)-basis of
\(\CC[[q]][[t]]\). Hence the expansion in this basis is unique. Using the algebraic identity
\begin{align*}
(\overleftarrow{\bk}*_+\overleftarrow{\bl})^{\dagger}&=(\overleftarrow{({}_{\down}(\overleftarrow{\bk}*_+\overleftarrow{\bl})_{\down})^{\vee}},2)\\
&=(({}_{\down}(\bk*_+\bl)_{\down})^{\vee},2)
\end{align*}
and comparing the coefficient of
\[
\frac{(q;q)_N}{(tq;q)_N}t^N,
\]
we obtain
\begin{align*}
&\sum_{\substack{0<n_1<\cdots<n_{a'}\\0<m_1<\cdots<m_{b'}\\n_{a'}+m_{b'}\leq N}}\left(\prod_{i=1}^{a'}\frac{q^{n_i(k_i'-1)}}{(1-q^{n_i})^{k_i'}}\right)\left(\prod_{i=1}^{b'}\frac{q^{m_i(l_i'-1)}}{(1-q^{m_i})^{l_i'}}\right)\\
&\qquad\cdot\frac{1-q^{n_{a'}}}{q^{n_{a'}}}\frac{1-q^{m_{b'}}}{q^{m_{b'}}}\frac{(q;q)_{N-n_{a'}-1}(q;q)_{N-m_{b'}-1}}{(q;q)_{N-1}(q;q)_{N-n_{a'}-m_{b'}}}\frac{q^N}{1-q^N}\\
&=\ol{z}_N(({}_{\down}(\bk*_+\bl)_{\down})^{\vee};q)\frac{q^N}{(1-q^N)^2}.
\end{align*}
Multiplying both sides by $q^{-N}(1-q^N)^2$ and taking the limit $q\to \zeta_N$, we obtain
\begin{align*}
&\ol{z}_N(({}_{\down}(\bk*_+\bl)_{\down})^{\vee})\\
&=\lim_{q\to \zeta_N}(1-q^N)\sum_{\substack{0<n_1<\cdots<n_{a'}\\0<m_1<\cdots<m_{b'}\\n_{a'}+m_{b'}\leq N}}\left(\prod_{i=1}^{a'}\frac{q^{n_i(k_i'-1)}}{(1-q^{n_i})^{k_i'}}\right)\left(\prod_{i=1}^{b'}\frac{q^{m_i(l_i'-1)}}{(1-q^{m_i})^{l_i'}}\right)\\
&\qquad\cdot\frac{1-q^{n_{a'}}}{q^{n_{a'}}}\frac{1-q^{m_{b'}}}{q^{m_{b'}}}\frac{(q;q)_{N-n_{a'}-1}(q;q)_{N-m_{b'}-1}}{(q;q)_{N-1}(q;q)_{N-n_{a'}-m_{b'}}}\\
&=0.
\end{align*}
This completes the proof.
\end{proof}

\section{Proof of Theorem \ref{thm:prove_conj}}

Let $F_+\coloneqq\bigoplus_{1\leq n, 2\leq k_1,\dots,k_n}\QQ(k_1,\dots,k_n)$. Tanaka and Wakabayashi proved an algebraic identity for cyclic sums.
We use the following special case.
\begin{Lem}[Tanaka--Wakabayashi {\cite[Proposition 2.6]{tw}}]\label{lem:1}
Let $(k_1,\dots,k_a)\in F_+$. Then
\[
\sum_{i=0}^{a-1}(k_{i+1},\dots,k_a,k_1,\dots,k_i)\equiv (-1)^{a-1}(k_1+\cdots+k_a)\pmod{F_+*F_+}.
\]
\end{Lem}
The proof of Lemma 3.1 uses only the recursive definition of the
quasi-shuffle product and the commutativity and associativity of the
underlying product on letters. Therefore, replacing the letter product $(r)\circ(s)=(r+s)$ by $(r)\circ_+(s)\coloneqq(r+s)+(r+s-1)$ in the proof, we obtain
\begin{align*}
\sum_{i=0}^{a-1}(k_{i+1},\dots,k_a,k_1,\dots,k_i)&\equiv (-1)^{a-1}(k_1\circ_+\cdots\circ_+k_a)\pmod{F_+*_+F_+}.
\end{align*}
Moreover, an induction on $a$ gives
\[
(k_1\circ_+\cdots\circ_+k_a)=\sum_{j=k-a+1}^{k}\binom{a-1}{k-j}(j),\qquad k\coloneqq k_1+\cdots+k_a.
\]
Combining these identities yields the following lemma.

\begin{Lem}\label{lem:2}
Let $\bk\coloneqq (k_1,\dots,k_a)\in F_+$,  and set $k=k_1+\cdots+k_a$. Then
\begin{align*}
\sum_{i=0}^{a-1}(k_{i+1},\dots,k_a,k_1,\dots,k_i)&\equiv (-1)^{a-1}\sum_{j=k-a+1}^{k}\binom{a-1}{k-j}(j)\pmod{F_+*_+F_+}.
\end{align*}
\end{Lem}

Define a $\QQ$-linear map $f:F_+\to\CC$ by
\[
f(\bk)\coloneqq \overline{z}_N(({}_{\down}\bk_{\down})^{\vee}).
\]

Theorem \ref{thm:prove_conj} follows from the following stronger statement.
\begin{Thm}
Let $(k_1,\dots,k_a)\in F_+$, and set $k\coloneqq k_1+\cdots+k_a$. Then
\begin{align*}
\sum_{i=0}^{a-1}f(k_{i+1},\dots,k_a,k_1,\dots,k_i)&=\frac{(-1)^{a-1}}{N}\binom{N+a-1}{k-1}.
\end{align*}
\end{Thm}

\begin{proof}
By Theorem \ref{thm:main0} and Lemma \ref{lem:2}, we have

\begin{align*}
\sum_{i=0}^{a-1}f(k_{i+1},\dots,k_a,k_1,\dots,k_i)&=(-1)^{a-1}\sum_{j=k-a+1}^k\binom{a-1}{k-j}f(j).
\end{align*}
For \(k-a+1\le j\le k\), \cite[Theorem 1.1]{btt2} gives
\[
f(j)=\ol{z}_N(\{1\}^{j-2})=\frac 1{N}\binom{N}{j-1}.
\]
Hence, Vandermonde's identity yields
\begin{align*}
\sum_{i=0}^{a-1}f(k_{i+1},\dots,k_a,k_1,\dots,k_i)&=\frac{(-1)^{a-1}}{N}\sum_{j=k-a+1}^k\binom{a-1}{k-j}\binom{N}{j-1}\\
&=\frac{(-1)^{a-1}}{N}\binom{N+a-1}{k-1}.
\end{align*}
This completes the proof.
\end{proof}

To deduce Theorem \ref{thm:prove_conj}, set \(a=t+1\) and \(k_i=d_{i-1}+2\) for \(1\leq i\leq a\). Then \(k=r+2\), and, by the definition of the Hoffman dual, the terms on the left-hand side of Theorem 3.3 are precisely the quantities
\[
(1-\zeta_N)^{-r}z_N\bigl(\{1\}^{d_j},2,\{1\}^{d_{j+1}},2,\ldots,2,\{1\}^{d_{j+t}}\bigr)\qquad (0\leq j\leq t),
\]
where the subscripts are read modulo \(t+1\). Multiplying the identity in Theorem 3.3 by \((1-\zeta_N)^r\) proves Conjecture 1.4.

\end{document}